\documentclass[preprint,3p]{elsarticle}

\usepackage{lineno,hyperref}
\usepackage{graphicx}
\usepackage{todonotes}
\usepackage{amsmath}
\usepackage{amssymb}
\usepackage{amsthm}
\usepackage{epstopdf}
\usepackage{algorithmic}
\newtheorem{definition}{Definition}[section]
\newtheorem{theorem}[definition]{Theorem}

\newtheorem{lemma}[definition]{Lemma}
\newtheorem{proposition}[definition]{Proposition}

\newtheorem{remark}{Remark}[section]
\newtheorem{assumption}[definition]{Assumption}

\newcommand{\BLUE}[1]{\textcolor{blue}{#1}}

\usepackage{enumerate}
\usepackage{color}
\usepackage{slashed}
\numberwithin{equation}{section}
\journal{Journal of \LaTeX\ Templates}
\begin{document}

\begin{frontmatter}

	\title{A numerical method for reconstructing the potential in fractional Calder\'{o}n problem with a single measurement}

\author[1]{Xinyan Li}
\ead{xinyanli@email.sdu.edu.cn}

%

\address[1]{Research Center for Mathematics and Interdisciplinary Sciences, Shandong University, 266237,  Qingdao, Shandong, P.R. China}

\begin{abstract}
In this paper, we develop a numerical method for determining the potential in one and two dimensional fractional Calder\'{o}n problems with a single measurement. Finite difference scheme is employed to discretize the fractional Laplacian, and the parameter reconstruction is formulated into a variational problem based on Tikhonov regularization to obtain a stable and accurate solution. Conjugate gradient method is utilized to solve the variational problem. Moreover, we also provide a suggestion to choose the regularization parameter. Numerical experiments are performed to illustrate the efficiency and effectiveness of the developed method and verify the theoretical results.
\end{abstract}

\end{frontmatter}

\section{\bf Introduction}

In this paper, we provide a numerical method to reconstruct the potential for fractional Calder\'{o}n problem with a single measurement. Fractional Calder\'on problem was proposed in \cite{GhoSal}, in which the fractional Schr\"odinger equation
\begin{equation}\label{fseghosh}
\left\{\begin{aligned}
((-\Delta)^s+q(x))u(x)&=0,\quad& &x\in\Omega,\\u(x)&=f(x),\quad& &x\in\Omega_e
\end{aligned}\right.
\end{equation}
is considered, where $s\in(0,1)$, $\Omega\subset\mathbb{R}^n$ is a bounded open set, $\Omega_e=\mathbb{R}^n\setminus\overline{\Omega}$, and $n\geq1$. In \eqref{fseghosh}, $(-\Delta)^su(x)$ can be defined by
\begin{equation}\label{def1}
	(-\Delta)^{s}u(x)=c_{n,s}  \text{ P. V. }\int_{\mathbb{R}^n}\frac{u(x)-u(x')}{|x-x'|^{n+2s}}dx',
\end{equation}
where P. V. denotes the principal value integral, $|x-x'|$ denotes the Euclidean distance between $x$ and $x'$, and the constant $c_{n,s}$ is given as the formula below
\begin{equation}\label{constantdef}
c_{n,s}=\frac{2^{2s}s\Gamma(\frac{n}{2}+s)}{\pi^{\frac{n}{2}}\Gamma(1-s)}.
\end{equation}
Furthermore, if $u$ belongs to the Schwartz space of rapidly decaying function, by \cite{Kwa17}, the
fractional Laplacian can be equivalently defined through the Fourier transform
\begin{equation}\label{def2}
(-\Delta)^su(x)=\mathcal{F}^{-1}(|\xi|^{2s}\mathcal{F}(u)).
\end{equation}
We will reconstruct $q(x)$ in \eqref{fseghosh} with a single measurement $\left.(-\Delta)^su\right|_{W}, W\subset\Omega_e$ numerically according to the theoretical result in \cite{GhoshJFA}.

Nowadays, fractional partial differential equations have attracted more and more attentions owing to successful applications in various fields such as quantum mechanics \cite{Las}, ground-water solute transport \cite{Ben13}, finance \cite{Guo18}, and stochastic dynamics \cite{MeeSik}. In practical applications, some important parameters in the equation are often difficult to be observed directly. Consequently, there has been increased focus on inverse problems related to fractional partial differential equations and their corresponding numerical methods. Among research on inverse problems of fractional partial differential equations, the time fractional equations have been widely investigated. For example, Cheng et al. studied the uniqueness of diffusion coefficient and fractional order in one-dimensional fractional diffusion equations \cite{CheNak}. Yamamoto and Zhang considered the conditional stability of a half-order fractional diffusion equation in determining a zeroth-order coefficient \cite{YamZha}. Sakamoto and Yamamoto analyzed the well-posedness of initial value/boundary value problems for fractional diffusion-wave equations, and some results about the uniqueness and stability are obtained \cite{SakYam}. Kirane and Malik investigated the existence and uniqueness of inverse source problems \cite{KirMal}. For more recent work on the theoreical aspect of time fractional inverse problems, see \cite{ZheChe,LiLiu,LiLuc}.  In terms of numerical methods of time fractional inverse problems, Liu and Yamamoto studied a backward problem for a time fractional partial differential equation to determine the initial status of the equation and implemented it numerically using a regularizing scheme\cite{LiuYam}. Zhang and Xu explored inverse source problems for time fractional diffusion equations and solved it numerically by assuming the source term as eigenfunction expansions series \cite{ZhaXu}. Sun and Wei provided the uniqueness for recovering the zeroth-order coefficient and fractional order of a time fractional diffusion equation simultaneously, and identified them numerically by introducing a Tikhonov variational problem and solving it through conjugate gradient method \cite{SunWei}. For more references, see \cite{BabBan,JiaLi,LiuWen}.

Recently, research on inverse problems for spatial fractional equations began to appear. Among them, fractional Calder\'on problem has been received widespread attention. Fractional Calder\'on problem is a generalization of Calder\'on problem \cite{Cal06} considering potential reconstruction in fractional Schr\"odinger equation \eqref{fseghosh}. In \cite{GhoSal}, Ghosh et al. constructed a Dirichlet-Neumann (DtN) map, and proved the unique determination of $q$ using the data of DtN map. Subsequently, R\"{u}land et al. studied the stability of the problem, gave the conclusion of logarithmic stability \cite{RulSala}, and proved the optimality of logarithmic stability in the literature \cite{RulSalb}. Based on the extension property of fractional Laplacian concluded by Caffarelli and Silvestre\cite{Caff}, and the analytic continuation property of the elliptic equation \cite{Hoem15}, Ghosh et al. proved that given a single $f$ and corresponding observation $\left.(-\Delta)^su\right|_{W}, W\subset\Omega_e$, one can uniquely reconstruct the potential $q$ \cite{GhoshJFA}. Compared with classical Calder\'on problem, this conclusion about fractional Calder\'on problem is quite different. Inverse problems about generalized form of fractional Schr\"odinger equations \eqref{fseghosh} have been studied over the past few years. For example, parameter reconstruction of the anisotropic fractional Schr\"odinger equation \cite{GhoLin,CaoLiu}, the nonlinear or unsteady fractional Schr\"odinger equation \cite{LaiLina,LaiLinb,KowLin}, and the fractional Schr\"odinger equation with the perturbation term \cite{CovMoen,Covi22,BhaGho}, and so on.

However, we would like to remark that research on numerical methods of fractional Calder\'on problem and its corresponding generalizations have not been paid much attention. The numerical schemes for the integral fractional Laplacian \eqref{def1} include finite difference method \cite{DuoZha,HuaObe,Li23}, finite element method \cite{BonBoe,AcoBor}, spectral method\cite{SheShe}, and spherical mean function method \cite{XuChe}. These methods will bring assistance in solving equations containing fractional Laplacian such as fractional Schr\"odinger equation. It is noticeable that most research on the equations with fractional Laplacian has focused on numerical methods of solving the forward problem, and inverse problems for these equations have become fruitful topics that offer great potential.

In this paper, we shall deal with an inverse problem of determining the potential term of fractional Schr\"odinger equation numerically. While the uniqueness of potential reconstruction in fractional Schr\"odinger equation with a single measurement was studied in \cite{GhoshJFA}, its numerical method has not been involved. We shall focus on the numerical methods and provide efficient numerical inversions with a numerical stability theory. The main contribution of this work is threefold: 
\begin{itemize}
\item We develop a fast finite difference method for two-dimensional fractional Sch\"odinger equations with inhomogeneous boundary conditions by truncating the computational domain, and error estimates are given to show the balance between discretization and truncation error;
\item We present a numerical method to reconstruct the potential in fractional Calder\'on problem with a single measurement for both one and two dimensional cases, which is achieved by employing conjugate gradient method to solve the given variational problem; 
\item We give a selection criterion for the regularization parameter and derive a logarithmic stability estimation. Numerical results corroborate the theoretical findings.
\end{itemize}

The following condition will be assumed in this paper:
\begin{assumption}\label{ass1}
If $u\in H^s(\mathbb{R}^n)$ solves
\begin{equation*}
\left\{\begin{aligned}
((-\Delta)^s+q(x))u(x)=0,\quad &x\in\Omega,\\u(x)=0,\quad &x\in\Omega_e,
\end{aligned}\right.
\end{equation*}
then $u\equiv0$. This is equivalent to the assumption that zero is not a Dirichlet eigenvalue of $(-\Delta)^s+q$ \cite{GhoSal}.
\end{assumption}

The rest of the paper is organized as follows. In Section 2, we discuss the well-posedness of the observation in order to establish a variational problem, and give a numerical method for two-dimensional fractional Schr\"odinger equation. In Section 3, numerical algorithms to reconstruct the potential are provided for both one and two dimensional cases by formulating the inverse problem into a variational problem, and we give some advice on parameter selection criterion with a logarithmic stability estimation. Numerical results are performed in Section 4 to illustrate the computational performance and verify the logarithmic stability estimation under the parameter selection criterion in our article. Finally, some concluding remarks are made in Section 5.

\section{\bf Overview and algorithms of the forward problem}
In this section, we analyse the well-posedness of the observation and provide numerical schemes for forward problem.

\subsection{Overview of the well-posedness}

In the following, we will use several $L^2$ based Sobolev spaces defined as \cite{McLean_2000}
\begin{equation*}
\begin{aligned}
&H^{\mu}(\mathbb{R}^n):=\left\{u \in L^2(\mathbb{R}^n); \int_{\mathbb{R}^n}\left(1+|\xi|^2\right)^{\mu}|\widehat{u}(\xi)|^2 d \xi<\infty\right\},\\
&\tilde{H}^{\mu}(U):=\text{closure of }C_c^{\infty}(U)\text{ in }H^{\mu}(\mathbb{R}^n),\\
&H^{\mu}(U):=\{u|_U:u\in H^{\mu}(\mathbb{R}^n)\},\\
\end{aligned}
\end{equation*}
where $U\subset\mathbb{R}^n$ is an open domain.

First, we recall a lemma about the existence, uniqueness and well-posedness of the weak solution of equation \eqref{fseghosh}.
\begin{lemma}\cite[Lemma 2.3.]{GhoSal}\label{exiuni}
Let $n\geq1,s\in(0,1)$, $\Omega$ be a bounded open set and $q\in L^{\infty}(\Omega)$. Suppose that Assumption \ref{ass1} holds. Let
$$B_q(v,w)=((-\Delta)^{\frac{s}{2}}v,(-\Delta)^{\frac{s}{2}}w)_{\mathbb{R}^n}+(qv|_{\Omega},w|_{\Omega})_{\Omega} ,v,w\in H^s(\mathbb{R}^n).$$
Then for any $f\in H^s(\mathbb{R}^n)$ and $g\in (\tilde{H}^s(\Omega))^*$, there is a unique solution $u\in H^s(\mathbb{R}^n)$ in \eqref{fseghosh} satisfying
$$B_q(u,w)=(g,w),\text{ for }w\in\tilde{H}^s(\Omega), \qquad u-f\in\tilde{H}^s(\Omega). $$
Moreover, the norm estimate
$$\|u\|_{H^s(\mathbb{R}^n)}\leq C(\|g\|_{(\tilde{H}^s(\Omega))^*}+\|f\|_{H^s(\mathbb{R}^n)})$$
holds with $C$ independent of $g$ and $f$.
\end{lemma}
Next, we provide the regularity estimate of the observation, and further illustrate that using the $L^2$ norm to estimate the residual of the Tikhonov regularization functional is reasonable under certain conditions.
\begin{proposition}(The regularization of the observation)
Assume $u$ is the solution of \eqref{fseghosh}, $f\in C_c^{\infty}(\Omega_e)$, $q\in L^{\infty}(\Omega)$, $W_1,W_2\subset\Omega_e$, $\mathrm{supp}(f)\subset W_1$, $\overline{W_1}\cap\overline{\Omega}=\phi$, and $\overline{W_2}\cap\overline{\Omega}=\phi$.  Then $\left. (-\Delta)^su\right|_{W_2}\in L^2(W_2)$.
\end{proposition}
\begin{proof}
Let $u=u_I+\breve{f}$, where
\begin{equation}\label{uIf}
u_I=\left\{\begin{aligned}
u,\quad &x\in\Omega,\\0,\quad &x\in\mathbb{R}^n\setminus\Omega,
\end{aligned}\right. \text{ and }
\breve{f}=\left\{\begin{aligned}
0,\quad &x\in\overline{\Omega},\\f,\quad &x\in\Omega_e.
\end{aligned}\right.
\end{equation}
For $f\in C_c^{\infty}(\Omega_e)$, one obtains that $(-\Delta)^s\breve{f}(x)\in L^2(W_2)$.
Moreover, based on Lemma \ref{exiuni}, $u\in H^s(\mathbb{R}^n)$, so that $u_I\in L^2(\Omega)$. For $\mathrm{dist}(\partial\Omega,\partial W_2)>0$, there is
\begin{equation*}
\begin{aligned}
\|(-\Delta)^su_I\|_{L^2(W_2)}^2&=\int_{W_2}\left(c_{n,s}  \text{ P. V. }\left(\int_{\mathbb{R}^n}\frac{u_I(x)-u_I(x')}{|x-x'|^{n+2s}}dx'\right)\right)^2dx\\
&=\int_{W_2}\left(c_{n,s}  \text{ P. V. }\left(\int_{\Omega}\frac{-u_I(x')}{|x-x'|^{n+2s}}dx'\right)\right)^2dx\\
&\leq C\int_{W_2}\left(\int_{\Omega}(u_I(x'))^2dx'\right)\left(\int_{\Omega}\frac{1}{|x-x'|^{2n+4s}}dx'\right)dx\\
&\leq C\int_{W_2}\left(\int_{\Omega}\frac{1}{|x-x'|^{2n+4s}}dx'\right)dx\\
&<\infty,
\end{aligned}
\end{equation*}
where the inequality holds for H\"older inequality, and $C$ is a positive constant. Due to the linearity of $L^2$ normed space, we obtain $\left. (-\Delta)^su\right|_{W_2}\in L^2(W_2)$.
\end{proof}
\subsection{The finite difference scheme for forward problem}\label{ss2_2}
In general,  one seeks the solution of a variational problem by solving the forward problem and updating the target. As a result, before performing inversion algorithms, we need to provide a suitable numerical method for the forward problem. For the reason that the external term $f$ in \cite[Theorem 1.]{GhoshJFA} is truncated in $W_1\in \Omega_e$, we consider the difference scheme of the equation
\begin{equation}\label{fsee2}
	\left\{
	\begin{aligned}
		(-\Delta)^su_R(x)+q(x)u_R(x) & =g(x), &  & x  \in (-L,L),                      \\
		u_R(x)            & =f(x), &  & x \in (-R,R)\setminus(-L,L),            \\
		u_R(x)            & =0,    &  & x \in \mathbb{R}\setminus(-R,R)
	\end{aligned}
	\right.
\end{equation}
by introducing the truncation parameter $R$, where $f(x)$ is positive and decay in the direction $|x|\to+\infty$. For finite difference scheme one-dimensional equation \eqref{fsee2}, we refer to \cite{Li23}. The two-dimensional truncated equation can be written as
\begin{equation}\label{fse2d}\hspace{-5mm}
	\left\{
	\begin{aligned}
		(-\Delta)^su_R(x,y)+q(x,y)u_R(x,y) & =g(x,y), & (x,y) \in (-L,L)^2,   \\
		u_R(x,y)&=f(x,y), &(x,y)\in(-R,R)^2\setminus(-L,L)^2,\\
        u_R(x,y)&=0,&(x,y)\in\mathbb{R}^2\setminus(-R,R)^2,
	\end{aligned}
	\right.
\end{equation}
where $s\in(0,1)$, $f(x,y)$ is positive and decay in the direction $(x^2+y^2)^{\frac{1}{2}}\to+\infty$. Now we give a finite difference scheme for \eqref{fse2d}. For a given positive integer $N$, we denote the space step size by $h=2L/N$, and then for $i,j\in\mathbb{Z}$, the discrete grid can be defined as $x_i=-L+ih,y_j=-L+jh$. Write $\xi=|x-x'|$, $\eta=|y-y'|$, $M=\lfloor\frac{R-L}{h}\rfloor+1(R>L)$, and define the domains and function
\begin{equation*}
\begin{aligned}
&D_1=(0,2L)^2, \\&D_2=(0,L+R)^2\setminus(0,2L)^2,\\
&D_3=\mathbb{R}_+^2\setminus (D_1\cup D_2)=\{(\xi,\eta)|\xi\geq L+R\text{ or }\eta\geq L+R\},\\
&\mathcal{U}_R(x,y,\xi,\eta)=(\xi^2+\eta^2)^{-(1+s)} \left(\sum_{k_1,k_2=0,1} u_R(x+(-1)^{k_1}\xi,y+(-1)^{k_2}\eta)-4u_R(x,y)\right).
\end{aligned}
\end{equation*}
The integral fractional Laplacian operator at the grid point $(x_i,y_j)(1\leq i,j\leq N-1)$ can be discretized to be
\begin{equation}\label{B12d}
\begin{aligned}
&(-\Delta)^su_R(x_i,y_j)\\=&-c_{2,s}\int_0^{\infty}\int_0^{\infty} \mathcal{U}_R(x_i,y_j,\xi,\eta)d\xi d\eta\\
=&-c_{2,s}\Bigg(\int_{D_1}\mathcal{U}_R(x_i,y_j,\xi,\eta)d\xi d\eta-4u_R(x_i,y_j)\int_{D_2\cup D_3}
\frac{1}{(\xi^2+\eta^2)^{1+s}}d\xi d\eta\\
 &\left. +\int_{D_2}(\xi^2+\eta^2)^{-(1+s)}\left(\sum_{k_1,k_2=0,1}f(x_i+(-1)^{k_1}\xi,y_j+(-1)^{k_2}\eta)\right)d\xi d\eta\right)\\
=&-c_{2,s}I_{i,j}+\mathcal{O}(h^{\kappa})-c_{2,s}(F_{R,i,j}+\epsilon_{R,i,j}).
\end{aligned}
\end{equation}
where $I_{i,j}$ is the discretization scheme of $\int_{D_1}\mathcal{U}_R(x_i,y_j,\xi,\eta)d\xi d\eta-4u_R(x_i,y_j)\int_{D_2\cup D_3}
\frac{1}{(\xi^2+\eta^2)^{1+s}}d\xi d\eta$, $\mathcal{O}(h^{\kappa})$ is the error of discretization, $F_{R,i,j}$ refers to the approximation of the integral
\begin{equation*}
\int_{D_2}(\xi^2+\eta^2)^{-(1+s)}\left(\sum_{k_1,k_2=0,1}f(x_i+(-1)^{k_1}\xi,y_j+(-1)^{k_2}\eta)\right)d\xi d\eta
\end{equation*}
by using numerical quadrature, the numerical quadrature error $\epsilon_{R,i,j}$ is a sufficiently small constant, and $c_{2,s}$ is defined by \eqref{constantdef}. Denote $u_{R,i,j}=u_R(x_i,y_j)$, and we have
\begin{equation*}
\begin{aligned}
I_{i,j}=&a_{00}u_{R,i,j}+\sum_{m=0}^{N}\left(\sum_{n=0\atop m+n\neq0}^{N}a_{mn}u_{R,i-m,j-n}+\sum_{n=1}^{N}a_{mn}u_{R,i-m,j+n}\right)\\ &+\sum_{m=1}^{N}\left(\sum_{n=0\atop m+n\neq0}^{N}a_{mn}u_{R,i+m,j-n}+\sum_{n=1}^{N}a_{mn}u_{R,i+m,j+n}\right),
\end{aligned}
\end{equation*}
where
$$a_{00}=-2\sum_{m=1}^{N}(a_{m0}+a_{0m})-4\sum_{m,n=1}^{N}a_{mn}-4\int_{\mathbb{R}^2_+\setminus D_1}(\xi^2+\eta^2)^{-(1+s)}d\xi d\eta,$$
and when $0\leq m,n\leq N,m+n\neq0$, we denote
\begin{equation*}
\begin{aligned}
a_{mn}=&\frac{2^{\sigma(m,n)}}{4((mh)^2+(nh)^2)^{\frac{\gamma}{2}}}\bigg(\int_{T_{mn}}(\xi^2+\eta^2)^{\frac{\gamma-(2+2s)}{2}}d\xi d\eta\\&\qquad\qquad\qquad+\bar{c}_{mn}\left\lfloor\frac{\gamma}{2}
\right\rfloor\int_{0}^{h}\int_{0}^{h}(\xi^2+\eta^2)^{\frac{\gamma-(2+2s)}{2}}d\xi d\eta\bigg),
\end{aligned}
\end{equation*}
where $\sigma(m,n)$ means the number of zeros of $m$ and $n$, the splitting parameter $\gamma\in(2s,2]$,
$$T_{mn}=(((m-1)h,(m+1)h)\times((n-1)h,(n+1)h))\cap[0,2L]^2,\quad 0\leq m,n\leq N, mn\neq0,$$
and
\begin{equation*}
\bar{c}_{mn}=\left\{\begin{aligned}
1,~~&m=0,n=1\text{ or }m=1,n=0,\\
-1,~~&m=1,n=1,\\
0,~~&\text{others}.
\end{aligned}\right.
\end{equation*}
The coefficient $\kappa$ depends on the regularity of the exact solution $u$ of the nonhomogeneous fractional Sch\"odinger equation
\begin{equation}\label{fse22}
	\left\{
	\begin{aligned}
		(-\Delta)^su(x,y)+q(x,y)u(x,y) & =g(x,y), & (x,y) \in (-L,L)^2,   \\
		u(x,y)&=f(x,y), &(x,y)\in\mathbb{R}^2\setminus(-L,L)^2\\
	\end{aligned}
	\right.
\end{equation}
and the value of splitting parameter $\gamma$. For $u\in C^{\lfloor 2s\rfloor,2s-\lfloor2s\rfloor+\varepsilon}(\mathbb{R}^2)$,  $0<\varepsilon\leq1+\lfloor2s\rfloor-2s$ and $\gamma\in(2s,2]$, $\kappa=\varepsilon$. For $u\in C^{2+\lfloor 2s\rfloor,2s-\lfloor2s\rfloor+\varepsilon}(\mathbb{R}^2)$, $0<\varepsilon\leq1+\lfloor2s\rfloor-2s$ and $\gamma=2$, $\kappa=2$.

Now we give the finite difference approximation of \eqref{B12d} as
\begin{equation*}
\begin{aligned}
&(-\Delta)_{h,\gamma}^su_{R,i,j}=-c_{2,s}\left[a_{00}u_{R,i,j}+\sum_{m=0}^{N}\left(\sum_{n=0\atop m+n\neq0}^{N}a_{mn}u_{R,i-m,j-n}+\sum_{n=1}^{N}a_{mn}u_{R,i-m,j+n}\right)\right. \\ &\left. +\sum_{m=1}^{N}\left(\sum_{n=0\atop m+n\neq0}^{N}a_{mn}u_{R,i+m,j-n}+\sum_{n=1}^{N}a_{mn}u_{R,i+m,j+n}\right)\right]-c_{2,s}F_{R,i,j},1\leq i,j\leq N-1,
\end{aligned}
\end{equation*}
where $u_{R,i,j}=f(x_i,y_j)$ when $i,j$ satisfy $(x_i,y_j)\in(-R,R)^2\setminus(-L,L)^2$, and $u_{R,i,j}=0$ when $(x_i,y_j)\in\mathbb{R}^2\setminus(-R,R)^2$.

For $1\leq i,j\leq N-1$, denote $U_{R,i,j}$ as the finite difference estimate of $u_{R,i,j}$, denote the vector and block vector
$$\begin{aligned}&\mathbf{u}_{R,x,j}=(u_{R,1,j},u_{R,2,j},\cdots,u_{R,N-1,j}), \mathbf{u}_{R,2}=(\mathbf{u}_{R,x,1},\mathbf{u}_{R,x,2},\cdots,\mathbf{u}_{R,x,N-1})^{\top},\\
&\mathbf{U}_{R,x,j}=(U_{R,1,j},U_{R,2,j},\cdots,U_{R,N-1,j}), \mathbf{U}_{R,2}=(\mathbf{U}_{R,x,1},\mathbf{U}_{R,x,2},\cdots,\mathbf{U}_{R,x,N-1})^{\top}.\end{aligned}$$
We write $u_{i,j}$ as the solution of fractional Schr\"odinger equation \eqref{fse22} on the point $(x_i,y_j)$, $1\leq i,j\leq N-1$. Define
$$\mathbf{u}_{x,j}=(u_{1,j},u_{2,j},\cdots,u_{N-1,j}), \mathbf{u}_{2}=(\mathbf{u}_{x,1},\mathbf{u}_{x,2},\cdots,\mathbf{u}_{x,N-1})^{\top}.$$
Similar as the deduction in \cite{DuoZha,Li23}, the local truncation error satisfies
$$\|(-\Delta)^s\mathbf{u}_2-(-\Delta)_{h,\gamma}^s\mathbf{u}_{R,2}\|\leq C(h^{\kappa}+R^{-2s}\max_{r\in\partial[-R,R]^2}f(r))+\epsilon_{2D},$$
where $C$ is a positive constant independent of $h$ and $R$, and $\epsilon_{2D}$ is an arbitrarily small positive constant related to numerical quadrature error. When the corresponding discrete matrix of the operator $(-\Delta)^s+q$ is positive definite, we have
\begin{equation}\label{equT8}
\|\mathbf{u}_2-\mathbf{U}_{R,2}\|\leq C(h^{\kappa}+R^{-2s}\max_{r\in\partial[-R,R]^2}f(r))+\epsilon_{2D}.
\end{equation}

\begin{remark}
Notice that the solution error in \eqref{equT8} can be divided into two parts, i.e.,  the discretization error $h^{\kappa}$ and the truncation error $R^{-2s}\max_{r\in\partial[-R,R]^2}f(r)$.  In order to balance them, we can choose $R$ such that $R^{-2s}\max_{r\in\partial[-R,R]^2}f(r)=\mathcal{O}(h^{\kappa})$, and thus obtain $||\mathbf{u}_2-\mathbf{U}_{R,2}||_{\infty}\leq Ch^{\kappa}$.
\end{remark}

\begin{remark}\label{accelerate}
It is worth noting that the matrices of our numerical methods are Toeplitz matrices for one-dimensional equations, and Toeplitz-block-block-Toeplitz matrices for two-dimensional equations. This implies that the system of linear algebraic equations could be efficiently solved by many well-developed Krylov subspace methods, where the matrix vector
multiplication operations can be efficiently performed using fast Fourier transform \cite{Ste,DuWan}.
\end{remark}
\section{\bf The inversion algorithm}
\subsection{One dimension inversion}
In this subsection, we assume that 
\begin{enumerate}[(i)]
\item Domain $\Omega\subset\mathbb{R}$ is a bounded open set satisfying strong local Lipschitz condition, $W_1,W_2\subset\Omega_e$ are open sets, and $\overline{\Omega}\cap \overline{W_1},\overline{\Omega}\cap\overline{W_2}=\emptyset$;
\item Functions $q$, $u$, $f$ satisfy the equation \eqref{fseghosh}, $q=0$ near $\partial\Omega$, and $f\in C_c^{\infty}(\Omega_e)\setminus\{0\}$, $\mathrm{supp}(f)\subset W_1$;
\item The observation $h^{\delta}$ satisfies
$$\|h^{\delta}(x)-(-\Delta)^su(x)\|_{L^2(W_2)}\leq\delta.$$
\end{enumerate}

Since the theoretical result in \cite{GhoshJFA} holds under the condition $q\in L^{\infty}(\Omega)$, we notice that when $\Omega\subset\mathbb{R}$ satisfies strong local Lipschitz condition, the embedding property $H^1(\Omega)\hookrightarrow L^{\infty}(\Omega)$ holds \cite{Adams}. Thus, the forward operator is introduced by denoting
\begin{equation*}
\begin{aligned}
F_1:H^1(\Omega)&\to L^2(W_2),\\ q&\mapsto\left. (-\Delta)^su\right|_{W_2}.
\end{aligned}
\end{equation*}
Besides, we introduce a Tikhonov regularization functional
\begin{equation}\label{var_pro}
J_1(q)=\frac{1}{2}||F_1(q)-h^{\delta}||_{L^2(W_2)}^2+\frac{\alpha}{2}||q'||^2_{L^2(\Omega)},
\end{equation}
where $\alpha$ is the regularization parameter. Here in the functional \eqref{var_pro}, the first part is used to control the value of $F_1(q)$ close to the measured data, and the second part is used to stabilize the derivative of the potential $q$. Then a feasible way to solve the inverse problem here is to solve the following minimization problem
\begin{equation}\label{vp_min}
J_1(q_{\alpha}^{\delta})=\min_{q\in H^1(\Omega)}J_1(q).
\end{equation}
\begin{remark}
The proof of the existence of the minimizer in \eqref{vp_min} is similar to that in \cite{DinZhe,SunWei}. However, it is difficult to find the minimizer due to machine precision. We could handle it by setting appropriate stopping rule to let $J_1(q_{\alpha}^{\delta})$ be closer to $\inf J_1(q)$.
\end{remark}
When $W_1=W_2$, the convergence of the minimization problem has the following property.
\begin{theorem}\label{stabt1}
Let $\Omega\subset\mathbb{R}$ be bounded non-empty Lipschitz open set, $W_1\subset\mathbb{R}$ be open set,  $\overline{\Omega}\cap\overline{W_1}=\emptyset$, and let $s\in [\frac{1}{4}, \frac{1}{2}]$. Assume that $f$, $q_r$ are respectively the boundary term and potential term of \eqref{fseghosh}. In addition, suppose the following conditions hold:
\begin{enumerate}[(i)]
\item For $\epsilon>0$, $f$ is chosen by $f\in \tilde{H}^{s+\epsilon}(W_1)\setminus\{0\}$, and
\begin{equation*}
\frac{\|f\|_{H^s(W_1)}}{\|f\|_{L^2(W_1)}}\leq C
\end{equation*}
for $C>0$;
\item Suppose that $\mathrm{supp}(q_r),\mathrm{supp}(q_{\alpha}^{\delta})\subset\Omega'\Subset\Omega$, where $\Omega'\Subset\Omega$ means that $\overline{\Omega'}\subset\Omega$ and $\overline{\Omega'}$ is a compact subset of
$\mathbb{R}$ \cite{Adams}, let the real potential $q_r=0$ near $\partial\Omega$,
\begin{equation*}
F_1(q_r)=h_r:=\left. (-\Delta)^su_r\right|_{W_1},~~q_r\in H^1(\Omega),
\end{equation*}
where $u_r$ is the real solution of \eqref{fseghosh} with $q_r$, and consider that $q_{\alpha}^{\delta}\in H^1(\Omega)$ satisfies
\begin{equation*}
J_1(q_{\alpha}^{\delta})\leq \inf_{q\in H^1(\Omega)}J_1(q)+\delta^2;
\end{equation*}
\item Assume that the observation satisfies
\begin{equation*}
\|h^{\delta}-h_r\|_{H^{-s}(W_1)}\leq\delta.
\end{equation*}
\end{enumerate}
Let $\alpha>0$ be such that  $\alpha\sim\delta^2$ when $\delta\downarrow0$.
Then
\begin{equation*}
\|q_{\alpha}^{\delta}-q_r\|_{L^{\infty}(\Omega)}=\mathcal{O}(\omega(\delta)),
\end{equation*}
as $\delta\downarrow0$, where
$$\omega(\delta)=C_1|\log(C_1\delta)|^{-\nu},$$
and $\nu>0$, $C_1>1$ only depend on $\Omega,W_1,s,C,n,\|f\|_{H^{s+\epsilon}(W_1)},\|q_{\alpha}^{\delta}\|_{H^1(\Omega)}, \|q_r\|_{H^1(\Omega)}$.
\end{theorem}
\begin{proof}
The main idea of the proof could be found in \cite[Theorem 2.1.]{CheYam}, and here we will provide more details for better understanding. Despite the stability estimate given in \cite[Theorem 1.]{RulJMA} as
\begin{equation}\label{Rulstab}
\|q_{\alpha}^{\delta}-q_0\|_{L^{\infty}(\Omega)}\leq \omega(\|F_1(q_{\alpha}^{\delta})-F_1(q_0)\|_{H^{-s}(W_1)})
\end{equation}
requiring $q_r,q_{\alpha}^{\delta}\in C^{0,s}(\overline{\Omega})$ with $\mathrm{supp}(q_r),\mathrm{supp}(q_{\alpha}^{\delta})\subset\Omega'\Subset\Omega$, when $s\in(0,\frac{1}{2}]$, it also holds for $H^1(\Omega)$ since $H^1(\Omega)\hookrightarrow C^{0,s}(\overline{\Omega})$ by Sobolev embedding theorem. By \cite[Corollary 6.31.]{Adams}, and \cite[Theorem 3.30., Theorem 3.33.]{McLean_2000}, we can obtain the equivalence of $H^1$ norm and $H^1$ semi-norm $\|(\cdot)'\|_{L^2(\Omega)}$. According to \cite[Theorem 1.]{GhoshJFA}, if $q\in L^{\infty}(\Omega)$, $s\in[\frac{1}{4},1)$, then the potential $q$ is uniquely determined, so we restrict $s\in[\frac{1}{4},\frac{1}{2}]$. Set $\mathcal{M}=\|q_r'\|_{L^2(\Omega)}$. Since $J_1(q_{\alpha}^{\delta})\leq J_1(q)+\delta^2$ for $q\in H^1(\Omega)$, we obtain
\begin{equation*}
\begin{aligned}
\|F_1(q_{\alpha}^{\delta})-h^{\delta}\|^2_{L^2(W_1)}+\frac{\alpha}{2}\|(q_{\alpha}^{\delta})'\|^2_{L^2(\Omega)}&\leq \|F_1(q_r)-h^{\delta}\|^2_{L^2(W_1)}+\frac{\alpha}{2}\|q_r'\|^2_{L^2(\Omega)}+\delta^2\\
&=\|h_r-h^{\delta}\|^2_{L^2(W_1)}+\frac{\alpha}{2}\|q_r'\|^2_{L^2(\Omega)}+\delta^2\\
&\leq 2\delta^2+\frac{\alpha}{2}\mathcal{M}^2.
\end{aligned}
\end{equation*}
Hence
\begin{equation*}
\begin{aligned}
\|F_1(q_{\alpha}^{\delta})-h^{\delta}\|_{L^2(W_1)}&\leq\left(2\delta^2+\frac{\alpha}{2}\mathcal{M}^2\right)^{\frac{1}{2}},\\
\|(q_{\alpha}^{\delta})'\|_{L^2(\Omega)}&\leq\left(\frac{4\delta^2}{\alpha}+\mathcal{M}^2\right)^{\frac{1}{2}}.
\end{aligned}
\end{equation*}
When $\alpha\sim\delta^2$, choose $C_2\delta^2\leq\alpha\leq C_3\delta^2$, and then we get
\begin{equation*}
\begin{aligned}
\|F_1(q_{\alpha}^{\delta})-h^{\delta}\|_{L^2(W_1)}&\leq(2+\frac{C_3}{2}\mathcal{M}^2)^{\frac{1}{2}}\delta,\\
\|(q_{\alpha}^{\delta})'\|_{L^2(\Omega)}&\leq(4C_2^{-1}+\mathcal{M}^2)^{\frac{1}{2}}.
\end{aligned}
\end{equation*}
Therefore
\begin{equation*}
\begin{aligned}
\|F_1(q_{\alpha}^{\delta})-F_1(q_r)\|_{L^2(W_1)}&\leq \|F_1(q_{\alpha}^{\delta})-h^{\delta}\|_{L^2(W_1)}+\|h^{\delta}-F_1(q_r)\|_{L^2(W_1)}\\&\leq((2+\frac{C_3}{2}\mathcal{M}^2)^{\frac{1}{2}}+1)\delta,
\end{aligned}
\end{equation*}
and by the properties of Sobolev spaces \cite{McLean_2000}, it can be verified that $L^2(W_1)$ embeds to $H^{-s}(W_1)$ and
\begin{equation*}
\|F_1(q_{\alpha}^{\delta})-F_1(q_r)\|_{H^{-s}(W_1)}\leq C_4((2+\frac{C_3}{2}\mathcal{M}^2)^{\frac{1}{2}}+1)\delta,
\end{equation*}
for a positive constant $C_4$. By \eqref{Rulstab},
\begin{equation*}
\|q_{\alpha}^{\delta}-q_0\|_{L^{\infty}(\Omega)}\leq \omega(\|F_1(q_{\alpha}^{\delta})-F(q_0)\|_{H^{-s}(W_1)})\leq\mathcal{O}(\omega(\delta))
\end{equation*}
holds.
\end{proof}
\begin{remark}
The result of Theorem \ref{stabt1} is based on accurate information of the noise bound $\delta$ or a good prediction of it. Actually, several 
strategies can be employed to improve the robustness, e.g., extracting the information from big data by local average, carrying on multiple repeated observations, and preprocessing noise by certain adjoint embedding operators, see \cite{CheZha,ZhoLi,LiHub}.
\end{remark}
In order to solve the variational problem \eqref{var_pro}, we utilize conjugate gradient method to iteratively update $q^k$ for each step. Conjugate gradient method has been applied to various inverse problems \cite{DinZhe,SunWei,ZheDin}, and the key task of it is to deduce the gradient of $J_1$. Let $q$ be perturbed by a small amount $\delta q\in Q_1=\{q\in L^2(\Omega);\|q\|_{H^1(\Omega)}<\infty,q=0\text{ near }\partial\Omega\}$. Then the forward solution has a small change denoted by
$$u_{q+\delta q}-u_q=u_q'\delta q+\tilde{r}.$$
Denote $\BLUE{\varphi}=u_q'\delta q$, and it satisfies the following sensitive problem\\
~\\
\textbf{Sensitive Problem:}
\begin{equation}\label{sp}
\left\{
\begin{aligned}
 ((-\Delta)^s+q)\varphi&=-\delta q\cdot u_q, &~~x\in\Omega,\\
\varphi&=0, &~~x\in\Omega_e,
\end{aligned}
\right.
\end{equation}
and $\tilde{r}$ satifies
\begin{equation}\label{odeltaq}
\left\{
\begin{aligned}
 (-\Delta)^s\tilde{r}+(q+\delta q)\tilde{r}&=-\delta q\cdot \varphi, &x\in\Omega,\\
\tilde{r}&=0, &x\in\Omega_e.
\end{aligned}
\right.
\end{equation}
On the basis of \cite[Lemma 2.3.]{GhoSal}, we get the estimate
\begin{equation*}
\|\varphi\|_{H^{s}(\mathbb{R})}\leq \|\delta q\|_{H^1(\Omega)}\|u_q\|_{H^{s}(\mathbb{R})},
\end{equation*}
\begin{equation*}
\|\tilde{r}\|_{H^{s}(\mathbb{R})}\leq \|\delta q\|_{H^1(\Omega)}\| \varphi\|_{H^{s}(\mathbb{R})}\leq\|\delta q\|_{H^1(\Omega)}^2\|u_q\|_{H^{s}(\mathbb{R})}.
\end{equation*}
Thus $\tilde{r}=o(\delta q)$ is a higher order infinitesimal of $\delta q$, which can be ignored.

Following \eqref{var_pro}, we obtain
\begin{equation}\label{A1i}
\begin{aligned}
\delta J_1(q)=&J_1(q+\delta q)-J_1(q)\\
=&\int_{W_2}((-\Delta)^su(x)-h^{\delta}(x))(-\Delta)^s\varphi(x)dx+\alpha(q',\delta q')_{L^2(\Omega)}\\
&+o(||(-\Delta)^s\varphi||_{L^2(W_2)}+||\delta q||_{H^1(\Omega)}).
\end{aligned}
\end{equation}
Suppose that $v_q$ satisfies the following adjoint problem\\
~\\
\textbf{Adjoint Problem:}
\begin{equation}\label{ap}
\left\{
\begin{aligned}
 ((-\Delta)^s+q)v&=0, &x\in\Omega,\\
v&=(-\Delta)^s\left. u_q\right|_{W_2}-h^{\delta}, &x\in\Omega_e,
\end{aligned}
\right.
\end{equation}
thus
\begin{equation}
\begin{aligned}
&(-q\varphi-\delta q\cdot u_q,v_q)_{\Omega}+((-\Delta)^s\left. \varphi\right|_{W_2},(-\Delta)^s\left. u_q\right|_{W_2}-h^{\delta})_{W_2}\\
=&((-\Delta)^s\varphi,v_q)_{\mathbb{R}^n}=(\varphi,(-\Delta)^sv_q)_{\mathbb{R}^n}=(\varphi,-qv_q)_{\Omega},
\end{aligned}
\end{equation}
and
$$\int_{W_2}((-\Delta)^su(x)-h^{\delta}(x))(-\Delta)^s\varphi(x)dx=\int_{\Omega}\delta qu_q\cdot v_qdx. $$

Hence by (\ref{A1i}), it holds that
\begin{equation*}
\begin{aligned}
\delta J_1(q)=&J_1(q+\delta q)-J_1(q)\\
=&\int_{\Omega}\delta qu_q\cdot v_qdx+\alpha(q',\delta q')_{L^2(\Omega)}+o(||(-\Delta)^s\varphi||_{L^2(W_2)}+||\delta q||_{H^1(\Omega)})\\
=&(u_qv_q,\delta q)_{L^2(\Omega)}+\alpha(-q'',\delta q)_{L^2(\Omega)}+o(||(-\Delta)^s\varphi||_{L^2(W_2)}+||\delta q||_{H^1(\Omega)})\\
=&(u_qv_q,\delta q)_{L^2(\Omega)}+\alpha(-q'',\delta q)_{L^2(\Omega)}+o(||(-\Delta)^s\varphi||_{L^2(W_2)}+||\delta q||_{H^1(\Omega)})\\
=&(u_qv_q-\alpha q'',\delta q)_{L^2(\Omega)}+o(||(-\Delta)^s\varphi||_{L^2(W_2)}+||\delta q||_{H^1(\Omega)}).
\end{aligned}
\end{equation*}
Due to the assumption in \cite[Theorem 1.]{GhoshJFA} that $\overline{\Omega}\cap\overline{W_2}=\emptyset$, by \eqref{def1} and the definition of $H^s(\mathbb{R}^n)$ norm in \cite{McLean_2000}, it holds that $\|(-\Delta)^s\varphi\|_{L^2(W_2)}\leq C\|\varphi\|_{H^{s}(\mathbb{R})}$. By \cite[Lemma 2.3.]{GhoSal},
\begin{equation*}
\|(-\Delta)^s\varphi\|_{L^2(W_2)}\leq \|\delta q\|_{H^1(\Omega)}\|u_q\|_{H^{s}(\mathbb{R})},
\end{equation*}
hence $o(\|(-\Delta)^s\varphi\|_{L^2(W_2)})$ is also a higher order infinitesimal of $\delta q$.

Therefore, the gradient of $J_1(q)$ can be written by
\begin{equation}\label{grad}
J_{q}'=u_qv_q-\alpha q''.
\end{equation}

 Assume that $q^k$ is the $k$-th iteration approximate solution of $q(x)$. Then the updating formula is
$$q^{k+1}=q^k+\beta^kd^k,$$
where $\beta^k$ is the step size, and $d^k$ is the descent direction in the $k$-th iteration updated by
\begin{equation}\label{desc}
d^k=-J_{q^k}'+\gamma^kd^{k-1},d^0=-J_{q^0}',
\end{equation}
with the conjugate coefficient $\gamma^k$ calculated by
\begin{equation}\label{conj}
\gamma^k=\frac{||J_{q^k}'||_{L^2(\Omega)}^2}{||J_{q^{k-1}}'||_{L^2(\Omega)}^2}, \gamma^0=0.
\end{equation}
In order to choose an appropriate step size $\beta^k$, we compute
\begin{equation*}
\begin{aligned}
&J_1(q^k+\beta^kd^k)\\
\approx&\frac{1}{2}\int_{W_2}((-\Delta)^su_{q^k}(x)+\beta^k(-\Delta)^s\varphi^k-h^{\delta}(x))^2dx
+\frac{\alpha}{2}\int_{\Omega}((q^k)'+\beta^k(d^k)')^2dx,
\end{aligned}
\end{equation*}
where $\varphi^k$ is the solution of sensitive problem \eqref{sp}. Let
\begin{equation}
\begin{aligned}
\frac{dJ_1}{d\beta^k}\approx&\int_{W_2}((-\Delta)^su_{q^k}(x)+\beta^k(-\Delta)^s\varphi^k-h^{\delta}(x))(-\Delta)^s\varphi^kdx\\
&+\alpha((q^k)'+\beta^k(d^k)',(d^k)')_{L^2(\Omega)}=0.
\end{aligned}
\end{equation}
Then the step size $\beta^k$ is given by
\begin{equation}\label{ss}
\beta^k=\frac{\int_{W_2}((-\Delta)^su_{q^k}(x)-h^{\delta}(x))(-\Delta)^s\varphi^kdx+\alpha((q^k)',(d^k)')_{L^2(\Omega)}} {\int_{W_2}((-\Delta)^s\varphi^k)^2dx+\alpha((d^k)',(d^k)')_{L^2(\Omega)}}.
\end{equation}
The iteration steps of the conjugate gradient method are given by
\begin{enumerate}
\item Initialize $q^0=0$, and set $k=0$, $d^0=-J_{q^0}'$;
\item Solve the forward problem (\ref{fsee2}), where we set $q=q^k$, and denote the residual $(-\Delta)^s\left. u\right|_{W_2}-h^{\delta}$;
\item Solve the adjoint problem (\ref{ap}), and determine the gradient $J_{q^k}'$ by \eqref{grad};
\item Calculate the conjugate coefficient $\gamma^k$ by \eqref{conj}, and the descent direction $d^k$ by \eqref{desc};
\item Solve the sensitive problem (\ref{sp}) and obtain $\varphi^k$, where we take $\delta q=d^k$;
\item Calculate the step size $\beta^k$ by \eqref{ss};
\item Update the zero order term $q^k$ by formula $q^{k+1}=q^k+\beta^kd^k$;
\item Increase $k=k+1$, return to Step 2, and repeat the above procedures until a stopping criterion is satisfied.
\end{enumerate}

\subsection{Two dimension inversion}
In this subsection, we assume that 
\begin{enumerate}[(i)]
\item Domain $\Omega\subset\mathbb{R}^2$ is a bounded open set with $C^{1,1}$ boundary, $W_1,W_2\subset\Omega_e$ are open sets, and $\overline{\Omega}\cap \overline{W_1},\overline{\Omega}\cap\overline{W_2}=\emptyset$;
\item Functions $q$, $u$, $f$ satisfy the equation \eqref{fseghosh}, $q$ and $\Delta q$ are zero near $\partial\Omega$, and $f\in C_c^{\infty}(\Omega_e)\setminus\{0\}$, $\mathrm{supp}(f)\subset W_1$;
\item The observation $h^{\delta}$ satisfies
$$\|h^{\delta}(x)-(-\Delta)^su(x)\|_{L^2(W_2)}\leq\delta.$$
\end{enumerate}
In two-dimensional case, if $\Omega\subset\mathbb{R}$ satisfies strong local Lipschitz condition, then the embedding property $H^2(\Omega)\hookrightarrow L^{\infty}(\Omega)$ holds. Thus, the forward operator changes into
\begin{equation*}
\begin{aligned}
F_2:H^2(\Omega)&\to L^2(W_2),\\ q&\mapsto\left. (-\Delta)^su\right|_{W_2},
\end{aligned}
\end{equation*}
and the Tikhonov regularization functional turns to
\begin{equation}\label{var_pro2d}
J_2(q)=\frac{1}{2}||F_2(q)-h^{\delta}||_{L^2(W_2)}^2+\frac{\alpha}{2}||\Delta q||^2_{L^2(\Omega)}.
\end{equation}
In \eqref{var_pro2d}, the first part is used to control the value of $F_2(q)$ to be close to the measurement $h^{\delta}$, and the second part is used to stabilize the second-order derivative of the potential $q$. We aim to solve the minimization problem
\begin{equation}\label{vp_min2d}
J_2(q_{\alpha}^{\delta})=\min_{q\in H^2(\Omega)}J_2(q).
\end{equation}
The the convergence estimate can be proposed similarly as:
\begin{theorem}\label{stabt2}
Let $\Omega\subset\mathbb{R}^2$ be bounded open set with $C^{1,1}$ boundary, and $W_1$ be open set such that $\overline{\Omega}\cap\overline{W_1}=\emptyset$. Assume that $s\in[\frac{1}{4},1)$, and that $f$, $q_r$ are the boundary term and potential term of \eqref{fseghosh} respectively. Suppose that the subsequent conditions are satisfied,
\begin{enumerate}[(i)]
\item For $\epsilon>0$, $f$ is chosen by $f\in \tilde{H}^{s+\epsilon}(W_1)\setminus\{0\}$, and
\begin{equation*}
\frac{\|f\|_{H^s(W_1)}}{\|f\|_{L^2(W_1)}}\leq C
\end{equation*}
for $C>0$;
\item Let the real potential $q_r=0$ near $\partial\Omega$,
\begin{equation*}
F(q_r)=h_r:=\left. (-\Delta)^su_r\right|_{W_2},~~q_r\in H^2(\Omega),
\end{equation*}
where $u_r$ is the real solution of \eqref{fseghosh} with real potential $q_r$; Let $q_{\alpha}^{\delta}\in H^2(\Omega)$ satisfy
\begin{equation*}
J_2(q_{\alpha}^{\delta})\leq \inf_{q\in H^2(\Omega)}J_2(q)+\delta^2,
\end{equation*}
and we assume that $\mathrm{supp}(q_r),\mathrm{supp}(q_{\alpha}^{\delta})\subset\Omega'\Subset\Omega$; 
\item Assume that the observation satisfies
\begin{equation*}
\|h^{\delta}-h_r\|_{H^{-s}(W_2)}\leq\delta.
\end{equation*}
\end{enumerate}
 Let $\alpha>0$ be such that $\alpha\sim\delta^2$ when $\delta\downarrow0$.
Then
\begin{equation*}
\|q_{\alpha}^{\delta}-q_r\|_{L^{\infty}(\Omega)}=\mathcal{O}(\omega(\delta)),
\end{equation*}
as $\delta\downarrow0$, where
$$\omega(\delta)=C_1|\log(C_1\delta)|^{-\nu},$$
and $\nu>0$, $C_1>1$ only depend on $\Omega,W_1,s,C,n,\|f\|_{H^{s+\epsilon}(W_1)},\|q_{\alpha}^{\delta}\|_{H^2(\Omega)}, \|q_r\|_{H^2(\Omega)}$.
\end{theorem}
\begin{proof}
According to stability estimate given in \cite[Theorem 1.]{RulJMA} as
\begin{equation}\label{Rulstab2d}
\|q_{\alpha}^{\delta}-q_0\|_{L^{\infty}(\Omega)}\leq \omega(\|F_2(q_{\alpha}^{\delta})-F_2(q_0)\|_{H^{-s}(W_2)})
\end{equation}
it requires $q_r,q_{\alpha}^{\delta}\in C^{0,s}(\overline{\Omega})$ with $\mathrm{supp}(q_r),\mathrm{supp}(q_{\alpha}^{\delta})\subset\Omega'\Subset\Omega$. When $s\in(0,1)$, it also holds for $H^2(\Omega)\hookrightarrow C^{0,s}(\overline{\Omega})$ by Sobolev embedding theorem. The equivalence $H^2$ norm and $\|\Delta(\cdot)\|_{L^2(\Omega)}$ can be deduced through \cite[Lemma 9.17.]{GilTru} and \cite[Theorem 3.30., Theorem 3.33.]{McLean_2000}. Combining \cite[Theorem 1.]{GhoshJFA}, we restrict $s\in[\frac{1}{4},1)$. The remaining proof is analogous to the counterpart in Theorem \ref{stabt1}.
\end{proof}

The admissible set in this subsection is given by
$$Q_2=\{q\in L^2(\Omega);\|q\|_{H^2(\Omega)}<\infty,q\text{ and }\Delta q\text{ are zero near }\partial\Omega\}.$$
Let $q$ be perturbed by a small amount $\delta q\in Q_2$. The forward solution has a small change denoted by
$$u_{q+\delta q}-u_q=u_q'\delta q+\tilde{r}.$$
Denote $\varphi=u_q'\delta q$, it also satisfies the two-dimensional case of \eqref{sp}. Similar to the derivation of one-dimensional case, the gradient of $J_2(q)$ can be given by
\begin{equation}\label{grad2d}
J_{q}'=u_qv_q+\alpha \Delta^2q,
\end{equation}
where $u_q$, $v_q$ are the solution of the equation \eqref{fseghosh} and \eqref{ap} in two-dimensional case respectively.

Next, we use conjugate gradient method to solve variational problem \eqref{var_pro2d}. Assume that the updating formula is
$$q^{k+1}=q^k+\beta^kd^k,$$
with the step size $\beta^k$, and the descent direction $d^k$ updated by
\begin{equation}\label{desc2d}
d^k=-J_{q^k}'+\gamma^kd^{k-1},d^0=-J_{q^0}',
\end{equation}
where the conjugate coefficient $\gamma^k$ is calculated by
\begin{equation}\label{conj2d}
\gamma^k=\frac{||J_{q^k}'||_{L^2(\Omega)}^2}{||J_{q^{k-1}}'||_{L^2(\Omega)}^2}, \gamma^0=0.
\end{equation}
The step size is estimated as
\begin{equation}\label{ss2d}
\beta^k=\frac{\int_{W_2}((-\Delta)^su_{q^k}(x)-h^{\delta}(x))(-\Delta)^s\varphi^kdx+\alpha(\Delta q^k,\Delta d^k)_{L^2(\Omega)}} {\int_{W_2}((-\Delta)^s\varphi^k)^2dx+\alpha(\Delta d^k,\Delta d^k)_{L^2(\Omega)}}
\end{equation}
to make $\frac{dJ_2}{d\beta_k}\approx0$.
The iteration steps of the conjugate gradient method are basically similar as them in one-dimensional case:
\begin{enumerate}
\item Initialize $q^0=0$, and set $k=0$, $d^0=-J_{q^0}'$;
\item Solve the forward problem (\ref{fse2d}), where we set $q=q^k$, and denote the residual $(-\Delta)^s\left. u\right|_{W_2}-h^{\delta}$;
\item Solve the adjoint problem (\ref{ap}), and determine the gradient $J_{q^k}'$ by \eqref{grad2d};
\item Calculate the conjugate coefficient $\gamma^k$ by \eqref{conj2d}, and the descent direction $d^k$ by \eqref{desc2d};
\item Solve the sensitive problem (\ref{sp}) and obtain $\varphi^k$, where we take $\delta q=d^k$;
\item Calculate the step size $\beta^k$ by \eqref{ss2d};
\item Update the zero order term $q^k$ by formula $q^{k+1}=q^k+\beta^kd^k$;
\item Increase $k=k+1$, return to Step 2, and repeat the above procedures until a stopping criterion is satisfied.
\end{enumerate}

\section{\bf Numerical inversions}
In this section, we present numerical results for 1D and 2D fractional Calder\'on problem. The observation is assumed as
\begin{equation*}
h^{\delta}=\left. (-\Delta)^su\right|_{W_2}+\delta (2rand(size(h))-1),
\end{equation*}
where $rand$ generates random numbers uniformly distributed on $[0,1]$, $size(h)$ is the number of discrete points in $W_2$, and $\delta$ is the noise level of the observation.

The stopping rule in the iteration algorithm is given as
\begin{equation*}
E_{k+1}=\|h^{\delta}-(-\Delta)^su_{k+1}\|_{L^2(W_2)}^2\leq 2\times\delta^2
\end{equation*}
in 1D problem, and
\begin{equation*}
E_{k+1}=\|h^{\delta}-(-\Delta)^su_{k+1}\|_{L^2(W_2)}^2\leq 10\times\delta^2
\end{equation*}
in 2D problem. In the light of Theorem \ref{stabt1} and Theorem \ref{stabt2}, we choose the regularization parameters that follow $\alpha\sim\delta^2$. When it reaches the stopping rule, we denote reconstruction solution $q_{\alpha}^{\delta}$. In the following examples, we use high-precision grids to calculate the observation by methods in Section \ref{ss2_2}. 

{\bf Example 4.1.}
In this example, the potential $q$ is given by $q(x) = \sin(\pi x)$, and we assume that $\Omega=(-1,1),s=0.4$, and $f(x)$ is $\mathbf{1}_{W_1}$ polished function. Let observation set $W_2=(-3,-1-\epsilon)\cup(1+\epsilon,3)$, $0<\epsilon\ll1$, $W_1=W_2$. The discrete matrix of the operator $(-\Delta)^s+q$ for this example is positive definite after our test. Figure \ref{q1de1}(a) presents the reconstruction results of $q_{\alpha}^{\delta}$ with noise level $\delta=${ 1E-7, 1E-5, 1E-3} compared to the real potential $q$, where $\alpha$ is given by $\alpha=\delta^2$. It illustrates that the reconstruction performs well for lower noise level $\delta$, however, the quality of the results deteriorates rapidly for larger noise, which reflects the ill-posedness of fractional Calder\'on problem. In order to show the stability of our algorithm, we choose noise level $\delta$=1E-7, 1E-6, 1E-5, 1E-4, 1E-3 and regularization parameter $\alpha=\delta^2$ to plot the relationship between $|log(\delta)|^{-1}$ and inversion error $\|q_{\alpha}^{\delta}-q\|_{L^{\infty}(\Omega)}$ in Figure \ref{q1de1}(b), from which we see that it is close to proportional relationship and validates the logarithmic stability result Theorem \ref{stabt1}.
\begin{figure}[htbp] 
\begin{minipage}[b]{.48\linewidth}
  \centering
  \centerline{\includegraphics[width=6.0cm]{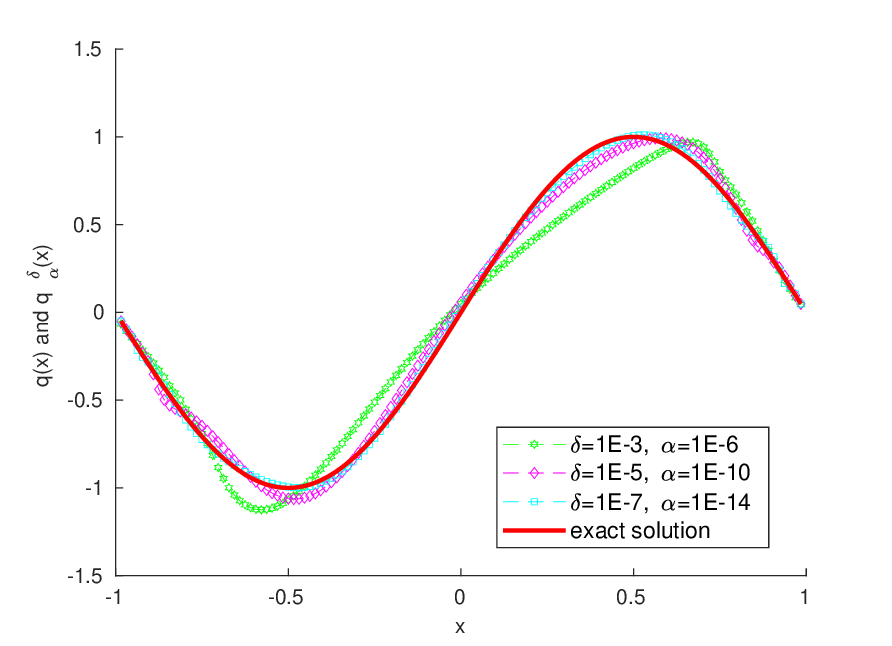}}
  \centerline{(a)}\medskip
\end{minipage}
\hfill
\begin{minipage}[b]{0.48\linewidth}
  \centering
  \centerline{\includegraphics[width=6.0cm]{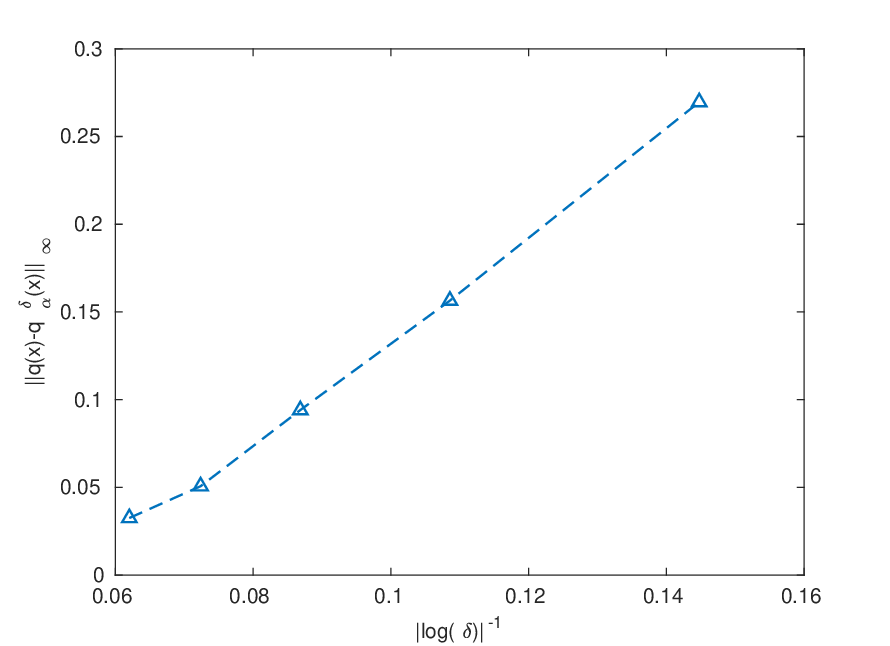}}
  \centerline{(b)}
  \medskip
\end{minipage}
\caption{The numerical results for Example 4.1 with different noise level $\delta$, (a) the reconstruction results $q_{\alpha}^{\delta}(x)$ and real solution $q(x)$; (b) the relationship between $|log(\delta)|^{-1}$ and $\|q_{\alpha}^{\delta}-q\|_{L^{\infty}(\Omega)}$.}
\label{q1de1}
\end{figure}

{\bf Example 4.2.}
In this example, $q(x)$ is prescribed as $q(x) = 10(0. 75-x)^2_+$. Other parameters such as $\Omega,s,f(x),W_1,W_2$ are the same as Example 4.1, and $\alpha$ is given by $\alpha=\delta^2$. The inversion potential $q_{\alpha}^{\delta}$ with noise level $\delta=${\rm 1E-7,1E-6,1E-5} and $\alpha=\delta^2$ is shown in Figure \ref{q1de2}(a) along with the real potential $q$. We observe that the inversion result is related to the noise level. When the noise level increases, the inversion results quickly get worse. Figure \ref{q1de2}(b) shows the relationship between $|log(\delta)|^{-1}$ and reconstruction error $\|q_{\alpha}^{\delta}-q\|_{L^{\infty}(\Omega)}$ with $\delta=$ 1E-7, 1E-6, 1E-5, 1E-4 and $\alpha=\delta^2$, which validates the logarithmic stability in Theorem \ref{stabt1}. As a result, the severe ill-posedness of the reconstruction can be further verified.
\begin{figure}[htbp]
\begin{minipage}[b]{.48\linewidth}
  \centering
  \centerline{\includegraphics[width=6.0cm]{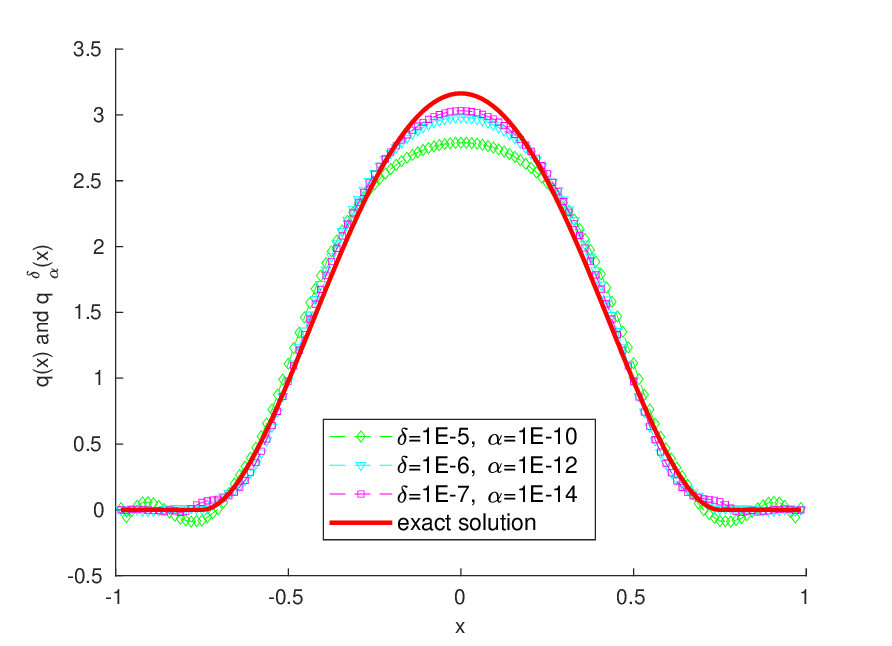}}

  \centerline{(a)}\medskip
\end{minipage}
\hfill             
\begin{minipage}[b]{0.48\linewidth}
  \centering
  \centerline{\includegraphics[width=6.0cm]{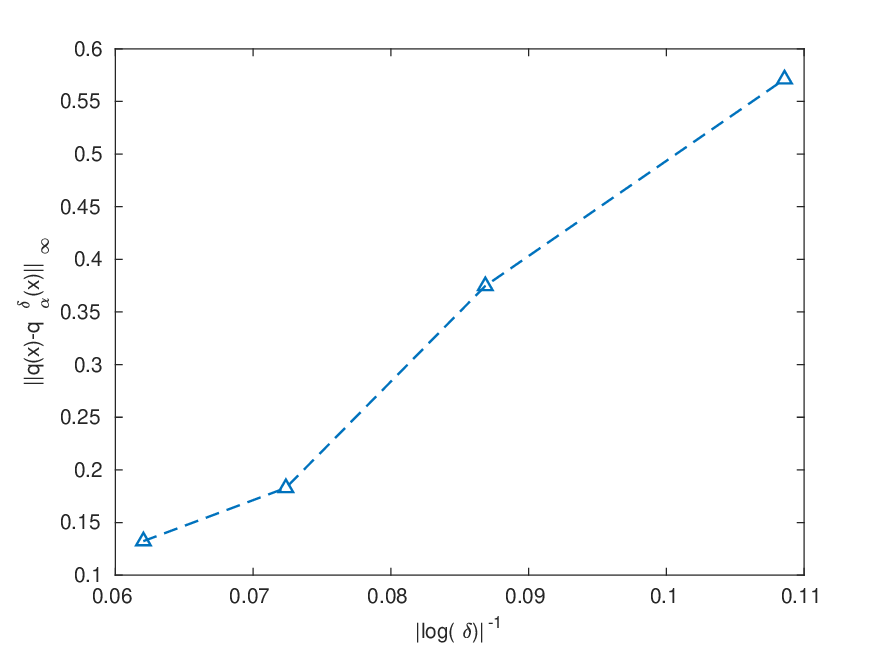}}

  \centerline{(b)}
  \medskip
\end{minipage}
\caption{The numerical results for Example 4.2 with different noise level $\delta$, (a) the reconstruction results $q_{\alpha}^{\delta}(x)$ and real solution $q(x)$; (b) the relationship between $|log(\delta)|^{-1}$ and $\|q_{\alpha}^{\delta}-q\|_{L^{\infty}(\Omega)}$.}
\label{q1de2}
\end{figure}

{\bf Example 4.3.}
Noticing that the uniqueness holds true for $q\in L^{\infty}(\Omega)$ in \cite{GhoshJFA}, in this example, we consider reconstructing a less regular piecewise constant $q(x)$ given as
\begin{equation*}
\left\{\begin{aligned}&1,\quad x\in(-\frac{1}{2},\frac{1}{2}),\\
&0,\quad x\in(-1,-\frac{1}{2}]\cup[\frac{1}{2},1).
\end{aligned}\right.
\end{equation*}
Other parameters $\Omega,s,f(x),W_1,W_2$ are the same as Example 4.1. Figure \ref{q1d1e7e3} presents the numerical result of estimated potential with noise level $\delta=${\rm 1E-7} and regularization parameter $\alpha=${\rm 1E-14} along with the real potential using our algorithm. The error is larger compared to the previous two examples because of the smoothing nature of the prior. As a result,
different regularization methods and prior information are needed when dealing with discontinuous and nonsmooth potential, for example, TV regularization methods\cite{MueSil}, and we will not go into details here.
\begin{figure}[htbp] 
  \centering
  \centerline{\includegraphics[width=6.0cm]{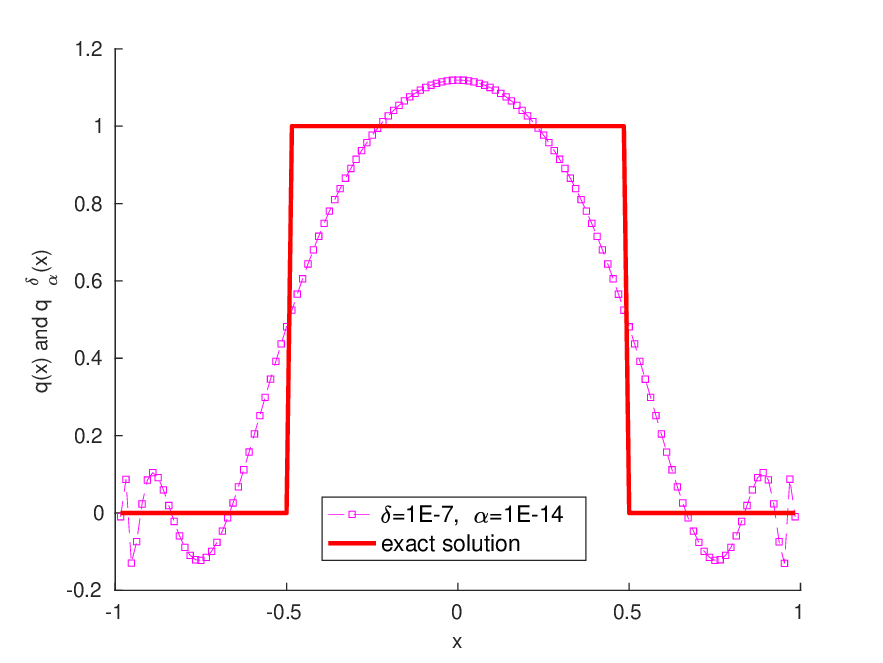}}
  \medskip
\caption{The numerical results of estimated potential for Example 4.3 with noise level $\delta={\rm 1E-7}$ and regularization parameter $\alpha={\rm 1E-14}$ along with the real potential $q(x)$.}
\label{q1d1e7e3}
\end{figure}

{\bf Example 4.4.}
This experiment tests the two-dimensional inversion. Let $q(x,y) = 100\times\max((0. 75^2-x^2)^3(0. 75^2-y^2)^3,0),\Omega=(-1,1)^2,s=0.5$, $f(x)$ is $\mathbf{1}_{W_1}$ polished function, the set $W_2=(-3,3)^2\setminus(-1-\epsilon,1+\epsilon)^2$, $0<\epsilon\ll1$ , and $W_1=W_2$. Figure \ref{q2de1}(a) and Figure\ref{q2de1}(b) show the real potential $q(x,y)$ and the reconstructed result $q_{\alpha}^{\delta}(x,y)$ with $\delta=${\rm 1E-6},$\alpha=${\rm 1E-13} respectively. The absolute error $|q_{\alpha}^{\delta}(x,y)-q(x,y)|$ is presented in Figure \ref{q2de1}(c) with $\delta=${\rm 1E-6},$\alpha=${\rm 1E-13}. It can be seen that the performance is worse on the points near the origin due to the maximum distance from the origin to $W_2$ and the zero initial value assumption in our algorithm. In Figure \ref{q2de1}(d), we select the noise level $\delta=$ 1E-6, 1E-5, 1E-4, 1E-3 and regularization parameter $\alpha=0.1\delta^2$ and draw the relationship between $|log(\delta)|^{-1}$ and $\|q_{\alpha}^{\delta}-q\|_{L^{\infty}(\Omega)}$. From the Figure \ref{q2de1} we shall see that it close to proportional relationship and verifies Theorem \ref{stabt2}.
\begin{figure}[htbp] 
{\begin{minipage}[b]{.48\linewidth}
  \centering
  \centerline{\includegraphics[width=6.0cm]{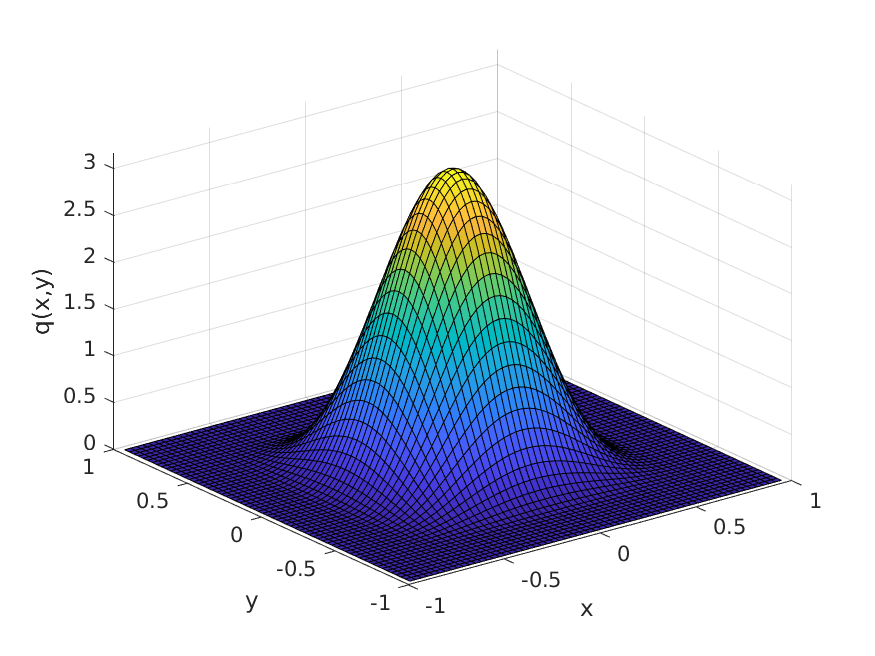}}
  \centerline{(a)}\medskip
\end{minipage}
\hfill              
\begin{minipage}[b]{0.48\linewidth}
  \centering
  \centerline{\includegraphics[width=6.0cm]{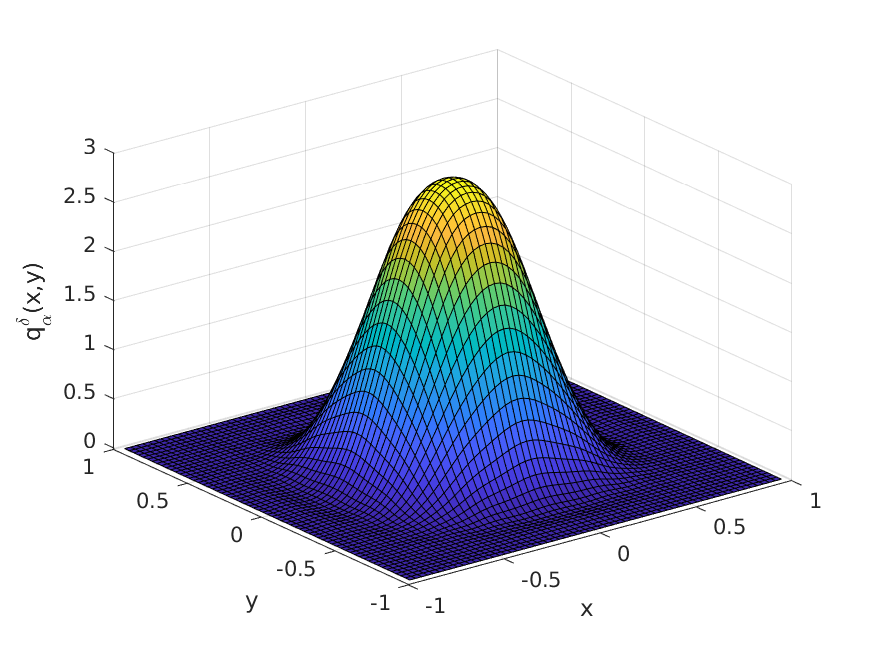}}
  \centerline{(b)}
  \medskip
\end{minipage}}
\vfill
{\begin{minipage}[b]{.48\linewidth}
  \centering
  \centerline{\includegraphics[width=6.0cm]{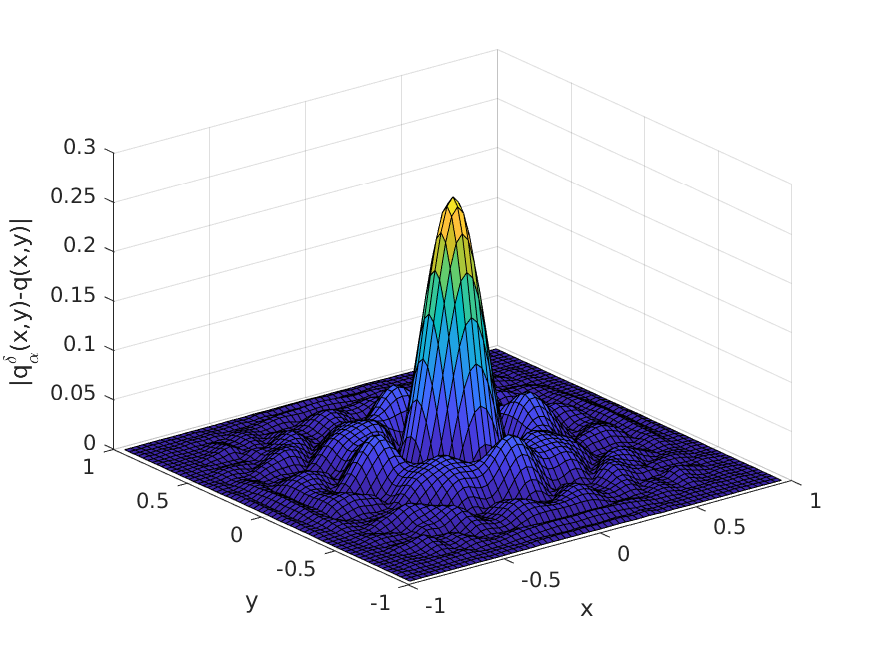}}
  \centerline{(c)}\medskip
\end{minipage}
\hfill              
\begin{minipage}[b]{0.48\linewidth}
  \centering
  \centerline{\includegraphics[width=6.0cm]{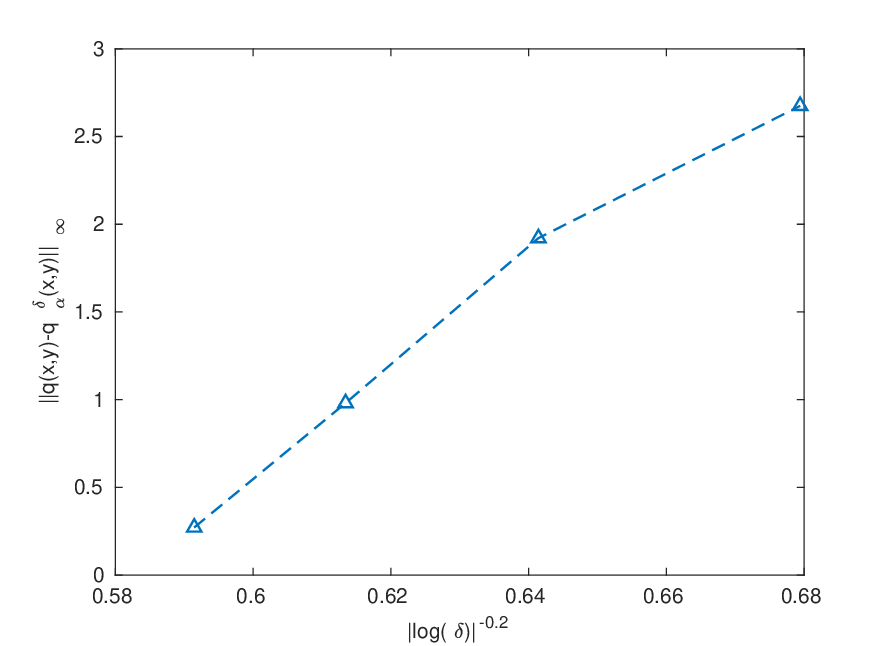}}
  \centerline{(d)}
  \medskip
\end{minipage}}
\caption{The numerical results of two-dimensional fraction Calder\'on problem for Example 4.4, (a) the real potential $q(x,y)$; (b) the reconstruction potential $q_{\alpha}^{\delta}(x,y)$ with $\delta$=1E-6,$\alpha$=1E-13; (c) the reconstruction error $|q_{\alpha}^{\delta}(x,y)-q(x,y)|$; (d) the relationship between $|log(\delta)|^{-1}$ and $\|q_{\alpha}^{\delta}-q\|_{L^{\infty}(\Omega)}$ with $\alpha=0.1\delta^2$.}
\label{q2de1}
\end{figure}

\section{\bf Conclusions}
In this work, by introducing the Tikhonov regularization functional, we propose a numerical method to reconstruct the potential for fractional Calder\'{o}n problem under a single measurement in  one-dimensional and two-dimensional cases. By choosing $\alpha\sim\delta^2$, we obtain the logarithmic stability. Conjugate gradient method is used to search the approximation of the regularized solution. The numerical experiments for both one-dimension and two-dimensional cases show the effectiveness of our proposed method.

  \end{document}